\documentclass[11pt]{amsart}
\usepackage{amsmath}
\usepackage{amssymb, verbatim}
\usepackage{pstricks,pst-node,pst-plot}
\usepackage{comment}
\usepackage{color}
\usepackage{extarrows}

\title[]{Concave elliptic equations and generalized Khovanskii-Teissier inequalities}
\author[T. C. Collins]{Tristan C. Collins}
  \email{tristanc@mit.edu}
  \address{Department of Mathematics, Massachusetts Institute of Technology, 77 Massachusetts Avenue, Cambridge, MA 02139}
 \thanks{T.C.C is supported in part by NSF grant DMS-1810924 and an Alfred P. Sloan Fellowship. }
  \dedicatory{Dedicated to D.H. Phong with admiration on the occasion of his 65th birthday.}

\theoremstyle{plain}
\newtheorem{thm}{Theorem}[section]

\newtheorem{defn}[thm]{Definition}
\newtheorem{lem}[thm]{Lemma}
\newtheorem{cor}[thm]{Corollary}

\theoremstyle{definition}
\newtheorem{ex}[thm]{Example}
\newtheorem{rk}[thm]{Remark}
\newtheorem{prob}[thm]{Problem}

\numberwithin{equation}{section}

\newcommand{\del}{\partial}

\newcommand{\dbar}{\overline{\del}}
\newcommand{\ddb}{\sqrt{-1}\del\dbar}

\renewcommand{\leq}{\leqslant}
\renewcommand{\geq}{\geqslant}

\renewcommand{\epsilon}{\varepsilon}
\renewcommand{\phi}{\varphi}

\begin{document}

\maketitle

\begin{abstract}
We explain a general construction through which concave elliptic operators on complex manifolds give rise to concave functions on cohomology.  In particular, this leads to generalized versions of the Khovanskii-Teissier inequalities.
\end{abstract}

\section{Introduction}

The purpose of this paper is to explain how concave elliptic equations on compact complex manifolds give rise to concave functions on (certain subsets of) ``cohomology".  For example, if $(X,\omega)$ is compact K\"ahler, we obtain concave functions on (certain subsets) of $H^{1,1}(X,\mathbb{R})$.  The prototype for the type function we consider is the volume function on K\"ahler classes.  Let $\mathcal{K}\subset H^{1,1}(X,\mathbb{R})$ be the open convex cone consisting of all $(1,1)$ classes admitting K\"ahler representatives.  For a K\"ahler class $\alpha \in \mathcal{K}$ define
\[
{\rm Vol}(\alpha) = \int_{X}\alpha^{n}
\]
Then the function $\alpha \mapsto {\rm Vol}(\alpha)^{1/n}$ is concave on $\mathcal{K}$ \cite{Dem, Kho, Gro, Teis1, Teis2, Teis3}.  In particular, for any K\"ahler classes $\alpha_0, \alpha_1$ we have
\begin{equation}\label{eq: introBM}
{\rm Vol}(\alpha_0+\alpha_1)^{1/n} \geq {\rm Vol}(\alpha_0)^{1/n} + {\rm Vol}(\alpha_1)^{1/n}.
\end{equation}
Furthermore, if $\beta \in H^{1,1}(X,\mathbb{R})$ is any class (not necessarily K\"ahler), then since $\mathcal{K}$ is open, $\alpha+t\beta \in \mathcal{K}$ for $t$ sufficiently close to zero.  Differentiating the volume function yields
\[
\left(\int_{X}\beta^2\wedge\alpha^{n-2}\right)\left(\int_{X}\alpha^{n}\right) \leq \left(\int_{X}\beta\wedge\alpha^{n-1}\right)^{2}.
\]
This latter inequality, originally discovered by Kohvanskii \cite{Kho} and Teissier \cite{Teis2, Teis3}, can be viewed as a generalization of the Hodge Index Theorem.  Many parts of the preceding discussion have in fact been extended beyond the K\"ahler cone, to the big cone where the definition of the volume function must be appropriately extended; see \cite{Bou, Laz, ELMNP, LM}.   

In this paper we will explain how the concavity of the volume function is a special case of a more general result which applies to a broad class of concave elliptic PDEs on complex manifolds.  For the time being we will be deliberately vague about what kind of classes we are considering, as they will vary between applications.  In any event,  to motivate our discussion, let us recall Demailly's proof of the Khovanskii-Teissier inequalities \cite{Dem}.  First, observe that for constant forms on a complex torus $\mathbb{C}^{n}\backslash \Lambda$, the concavity of the volume function is equivalent to the well-known concavity of the map
\[
{\rm Herm}_{+}(n) \ni M \longmapsto {\rm \det}(M)^{\frac{1}{n}} \in \mathbb{R}
\]
where ${\rm Herm}_{+}(n)$ denotes the set of positive definite, $n\times n$ Hermitian matrices.  Next, we apply Yau's solution of the Calabi conjecture \cite{Y} to reduce the global inequality to the pointwise inequality.  Namely, fix a K\"ahler form $\omega$, and assume that $\int_{X}\omega^{n}=1$.  By Yau's theorem \cite{Y} we can find a K\"ahler forms $\tilde{\alpha_i}$ with $[\alpha_i]=[\tilde{\alpha}_i]$ for $i=0,1$ satisfying
\begin{equation}\label{eq: introCMA}
\tilde{\alpha}_i^n = \left(\frac{[\alpha_i]^n}{[\omega]^n}\right) \omega^n.
\end{equation}
Applying the point-wise concavity we clearly have
\[
\left(\frac{(\tilde{\alpha}_0 + \tilde{\alpha}_1)^n}{\omega^n}\right)^{1/n} \geq \left(\frac{\tilde{\alpha}_0^n}{\omega^n}\right)^{1/n} +\left(\frac{\tilde{\alpha}_1^n}{\omega^n}\right)^{1/n}
\]
at each point $p\in X$.  Combining this with~\eqref{eq: introCMA} yields
\[
(\tilde{\alpha}_0 + \tilde{\alpha}_1)^n \geq \left[\left(\frac{[\alpha_0]^n}{[\omega]^n}\right)^{1/n} +\left(\frac{[\alpha_1]^n}{[\omega]^n}\right)^{1/n}\right]^n \omega^n.
\]
Integrating this inequality over $X$ yields~\eqref{eq: introBM}.  They key idea here is that the concavity of the elliptic operator $F(M)= \left(\det(M)\right)^{1/n}$ yields the concavity of the associated function on cohomology on the set of classes where the equation $F(M)= c$ is globally soluble on $(X,\omega)$.  There has recently been some interest in generalizing the Khovanskii-Teissier inequalities using different concave functions.  For some recent work in this direction, using the $\sigma_k$ operators, see \cite{Xia1, Xia2} and the references therein.


\section{Generalized Khovanskii-Teissier inequalities}

To begin with, we recall some of the basic definitions of the elliptic operators we will study.  Since our focus is on complex manifolds, some of our definitions will be suited to this particular case.

To begin with, fix an open, symmetric set $\Gamma \subset \mathbb{R}^n$, and denote by $\Gamma_n = \mathbb{R}^{n}_{+}$ the positive orthant.  Suppose that $f(\lambda_1, \ldots, \lambda_{n}):\Gamma\rightarrow \mathbb{R}$ is a smooth, symmetric function (in fact, one only needs $f$ to be $C^2$ for what follows, but we will not pursue this).   We suppose in addition that
\begin{itemize}
\item[(i)] $f$ is strictly increasing in each of its arguments.  That is $\frac{\del f}{\del \lambda_i} >0$.
\item[(ii)]$ f$ is concave.
\end{itemize}
By abuse of notation, denote by $\Gamma \subset {\rm Herm}(n)$ the set of Hermitian matrices whose eigenvalues lie in $\Gamma \subset \mathbb{R}^n$.  Then $f$ induces an elliptic operator on $\Gamma$ by
\[
F(M) = f(\lambda_1(M), \ldots, \lambda_n(M))
\]
where $\lambda_1(M), \ldots, \lambda_n(M)$ are the eigenvalues of $M$.  The function $F: \Gamma \rightarrow \mathbb{R}$ is concave and elliptic, in the sense the $F(M+P) > F(M)$ for all $P\in \Gamma_n$.  In all of our examples the operator $F$ will be defined on all of ${\rm Herm}(n)$.  Furthermore, in most (but not all examples) $\Gamma$ will be a cone in $\mathbb{R}^n$ with vertex at the origin, containing $\Gamma_n$.  For our purposes we will always think of an elliptic operator as a pair $(f, \Gamma)$ of an operator, together with an admissible cone. We pause to discuss a variety of examples.

\begin{ex}[Hessian Equations]\label{ex: hessEqPDE}  Let $\Gamma_k \subset \mathbb{R}^n$ be
\[
\Gamma_k = \{ (\lambda_1,\ldots, \lambda_n) : \sigma_{\ell}(\lambda_1, \ldots, \lambda_n) > 0 \text{ for all } 1 \leq \ell \leq k \}
\]
where $\sigma_{\ell}$ denotes the $\ell$-th symmetric function on $(\lambda_1, \ldots, \lambda_n)$.  Then it is well-known \cite{Gar, CNS} that on the cone $ \Gamma_{k}$, the function
\[
f(\lambda) = \left(\binom{n}{k}^{-1}\sigma_{k}(\lambda)\right)^{1/k}
\]
is concave, and increasing in each of its arguments.  In fact, slightly more is true \cite{Gar}; $\sigma_k^{1/k}$ is strictly concave in the following sense.  Let $M \in \Gamma_k$, and $B$ any Hermitian matrix.  Then
\[
\frac{d^2}{dt^2}\big|_{t=0}\sigma_k^{1/k}(M +tB) \leq 0
\]
with equality if and only if $B=\lambda M$ for some $\lambda \in \mathbb{R}$.
\end{ex}

\begin{ex}[Hessian Quotient Equations]\label{ex: hessQuoPDE}  One can also consider ratios of symmetric functions.  For $1\leq \ell <k \leq n$ we consider the function
\[
f(\lambda) = \left(\frac{\binom{n}{k}^{-1}\sigma_{k}(\lambda)}{\binom{n}{\ell}^{-1}\sigma_{\ell}(\lambda)}\right)^{\frac{1}{k-\ell}}
\]
This function is increasing in each of its arguments, and concave \cite{Spr}.  Note that the case $\ell=n-1, k= n$ is the operator corresponding to the $J$-equation, introduced by Donaldson \cite{Do} and Chen \cite{Chen04}.
\end{ex}

\begin{ex}[The Lagrangian phase operator]\label{ex: lagPhasePDE}
Consider the Lagrangian phase operator
\[
f(\lambda_{1},\ldots, \lambda_{n}) = \Theta(\lambda) := \sum_{i=1}^{n}\arctan(\lambda_i).
\]
This operator arises in the study of special Lagrangian submanifolds of Calabi-Yau manifolds, and the deformed Hermitian-Yang-Mills equation in mirror symmetry.  While $\Theta$ is increasing in all of its arguments on all of $\mathbb{R}^n$ it is well known that $\Theta$ is not globally concave on $\mathbb{R}^n$.  Instead, we have
\begin{enumerate}
\item $\Theta$ is concave on the set $\{(\lambda_1,\ldots, \lambda_n): \Theta(\lambda) \geq (n-1)\frac{\pi}{2} \}$ \cite{Yuan}
\item For every $\delta>0$ there is a constant $A= A(\delta)$ such that $-e^{-A\Theta}$ is concave on the set $\{(\lambda_1,\ldots, \lambda_n): \Theta(\lambda) \geq (n-2)\frac{\pi}{2}+\delta \}$ \cite{CPW}
\end{enumerate}
For $n\geq 3$, the operator $\Theta$ can fail to have even convex level sets in the range $(-(n-2)\frac{\pi}{2} +\delta, (n-2)\frac{\pi}{2} -\delta)$.
\end{ex}

\begin{ex}[The $(n-1,n-1)$ Hessian equations]\label{ex: n-1Eq}
For $(\lambda_1,\ldots, \lambda_n) \in \mathbb{R}^n$ we define a map $P: \mathbb{R}^n\rightarrow \mathbb{R}^n$, written as $P(\lambda_1,\ldots, \lambda_n) = (\mu_1,\ldots, \mu_n)$, by
\[
\mu_i= \frac{1}{n-1}\sum_{j \ne i} \lambda_j.
\]
Then we define the $(n-1,n-1)$ $k$-Hessian operator to be
\[
f(\lambda_1, \ldots, \lambda_n) = \left(\binom{n}{k}^{-1}\sigma_{k}(P(\lambda))\right)^{1/k}.
\]
This operator is clearly elliptic, and concave on the set $P^{-1}(\Gamma_k)$.
\end{ex}

An elliptic operator in the above sense gives rise to an elliptic operator on a complex manifold $(X,\omega)$ in the following way.  Given a smooth $(1,1)$ form $\alpha$ on $X$, we can associate to $\alpha$ a Hermitian endomorphism of $T^{1,0}X$ using the K\"ahler metric.  Namely, in local coordinates, define
\[
\mathcal{E}(\alpha)^{i}_{j} = \omega^{i\bar{k}}\alpha_{\bar{k}j}.
\]
Since $\mathcal{E}(\alpha)$ is Hermitian, it has $n$ real eigenvalues which we denote by $\lambda_1(\alpha), \ldots, \lambda_n(\alpha)$.  Note that, while we have suppressed the dependence, these eigenvalues depend on the choice of Hermitian metric $\omega$.  In order for our results to be broadly applicable, we make the following definition.
\begin{defn}\label{def: ansatz}
We define an {\em ansatz} to be a smooth, linear map $W: \Lambda^{1}_{\mathbb{R}}  \rightarrow \Lambda^{1,1}_{\mathbb{R}}$.  For a given ansatz $W$, and a smooth $(1,1)$ form $\alpha$ on $X$ we define the $W$-class of $\alpha$ by
\[
[\alpha]_{W} = \{ \alpha_{i\bar{j}} + \del_i\del_{\bar{j}}\phi + W(d\phi)_{i\bar{j}} : \phi \in C^{\infty}(X,\mathbb{R})\}
\]
\end{defn}

Note that, since $W$ is linear, it is easy to see that the set $[\alpha]_{W}\subset \Lambda^{1,1}(X,\mathbb{R})$ is an equivalence class; hence the terminology.  We endow the set of $W$-classes with the quotient topology induced by the $C^{4}$ norm on $\Lambda^{1,1}$.

\begin{rk}
Our primary interest is going to be the case when $W=0$, and $d\alpha=0$.  In this case $[\alpha]_W = [\alpha]_{BC}$, the Bott-Chern class of $\alpha$.
\end{rk}

\begin{defn}
Given an elliptic operator $(f, \Gamma)$ we will say that a form $\alpha$ is {\em admissible} (with respect to $\omega$) if $\lambda(\alpha) \in \Gamma$.  Furthermore, we will say that the $W$-class $[\alpha]_{W}$ is admissible if $[\alpha]_{W}$ contains an admissible representative.
\end{defn}
Note in particular that, since $\Gamma$ is open and $W$ is smooth, the set of admissible $W$-classes is open.  Concretely, this means that, for any smooth $(1,1)$ form $\beta$ the class $[\alpha+\epsilon\beta]_{W}$ is admissible for $\epsilon$ sufficiently small.  With this discussion in place, it makes sense to study the following very general elliptic problem on $(X,\omega)$.

\begin{prob}\label{prob: genEll}
Fix an elliptic operator $(f,\Gamma)$, and an ansatz $W$. For each admissible class $[\alpha]_{W}$, find a constant $c$, and a smooth, admissible $(1,1)$ form $\tilde{\alpha} \in [\alpha]_{W}$ such that
\begin{equation}\label{eq: genEll}
F_{\omega}(\tilde{\alpha}) = f(\lambda_1(\tilde{\alpha}), \ldots, \lambda_n(\tilde{\alpha})) = c
\end{equation}
\end{prob} 

\begin{defn}
Given an elliptic operator $(f, \Gamma)$ and an ansatz $W$ we will denote by ${\rm Sol}_{\omega}(f,\Gamma, W)$ the set of admissible $W$-classes admitting a solution to Problem~\ref{prob: genEll} for some constant $c$.  
\end{defn}

We are now going to prove two elementary structural results about solutions of Problem~\ref{prob: genEll}.  Our first result is that, on the set ${\rm Sol}_{\omega}(f,\Gamma, W)$, the constant $c$ appearing in the statement of Problem~\ref{prob: genEll} is unique in $[\alpha]_{W}$, and hence descends to a function
\[
 {\rm Sol}_{\omega}(f,\Gamma, W) \ni [\alpha]_{W} \longmapsto c(\omega,[\alpha]_{W}).
 \]
 The second structural result is that the set ${\rm Sol}_{\omega}(f,\Gamma, W)$ is either empty, or open in the set of $W$-classes.

\begin{lem}\label{lem: well-defined}
Fix a compact Hermitian manifold $(X,\omega)$, and a cohomology class $[\alpha]$.  We consider an ansatz $W$, and an equation $(f, \Gamma)$, and fix a $(1,1)$ form $\alpha \in [\alpha]_W$.  Assume that there are two functions $\phi_0, \phi_1$ such that
$\alpha_i := \alpha +\ddb \phi_i + W(d\phi)$ are admissible solutions of $F_{\omega}(\alpha_i)= c_i$ for $i=0,1$.  Then $c_0=c_1$.  In particular, on ${\rm Sol}_{\omega}(f,\Gamma, W)$ the function
\[
[\alpha]_W \mapsto c(\omega, [\alpha]_W)
\]
is well-defined.
\end{lem}
\begin{proof}
This is just an application of the maximum principle.  Since $X$ is compact, there are points $p_{min}$ (resp. $p_{max}$ )where $\phi_0-\phi_1$ attains its minimum (resp. maximum).  At $p_{min}$ we have
\[
d\phi_0 = d\phi_1, \qquad \ddb \phi_0 \geq \ddb \phi_1.
\]
It follows that $\alpha_0 \geq \alpha_1$.  Since $\alpha_i$ are both admissible, we can invoke the ellipticity of $f$ to obtain
\[
c_0 = F(\alpha_0) \geq F(\alpha_1) = c_1.
\]
Arguing similarly at the maximum implies $c_0=c_1$.
\end{proof}
\begin{rk}
Note that the above argument, together with the strong maximum principle, also implies uniqueness of solutions to Problem~\ref{prob: genEll} in a given $W$-class.
\end{rk}

\begin{lem}\label{lem: open}
Fix a compact complex manifold $(X,\omega)$, and an ansatz $W$. Consider an equation $(f, \Gamma)$.  Then, with notation as above, ${\rm Sol}_{\omega}(f,\Gamma, W)$ is either empty, or open.
\end{lem}
\begin{proof}
This is nothing but the implicit function theorem.  Assume that ${\rm Sol}_{\omega}(f,\Gamma, W)$ is not empty.  Let $[\alpha]_W$ admit a smooth solution $\alpha$ of the equation $F_{\omega}(\alpha) = c(\omega, [\alpha]_W)$.  Fix a $(1,1)$ form $\beta$, and consider the map
\[
\mathcal{F}(\phi, t, a) = F_{\omega}(\alpha+t\beta+\ddb \phi + W(d\phi)) -c(\omega,[\alpha]_{W})-a
\]
Since $\Gamma$ is open, there is small neighborhood $U$ of $(0,0,0) \in C^{k,\alpha}(\mathbb{R}) \times \mathbb{R} \times \mathbb{R}$ such that $F : U \rightarrow C^{k-2,\alpha}(\mathbb{R})$. Let
\[
F^{i\bar{j}} = \frac{\del F}{\del m_{i\bar{j}}}
\]
denote the derivative of $F$ at $\alpha$, and note that, since $F$ is elliptic, $F^{i\bar{j}}$ defines a Hermitian metric on $X$.  Let $\Delta_{F} = F^{i\bar{j}}\del_i\del_{\bar{j}}$ denote the associated Laplacian.  The operator
\[
v \mapsto Lv := \Delta_{F}v + F^{i\bar{j}}W_{i\bar{j}}(dv)
\]
is homotopic to $\Delta_{\omega}$ and therefore has index zero, regarded as a map from $C^{k,\alpha} \rightarrow C^{k-2,\alpha}$.  By the maximum principle, the kernel of $L$ consists of the constant functions.  Thus, the cokernel of $L$ has dimension $1$.  Another application of the maximum principle shows that the cokernel of $L$ contains the constants.  It follows that the map
\[
C^{k,\alpha} \times \mathbb{R} \ni (v,a) \longmapsto Lv -a \in C^{k-2,\alpha}
\]
is surjective.  On the other hand, the linearization of $\mathcal{F}$ is
\[
D\mathcal{F}(v, \beta, a) = Lv + F^{i\bar{j}}\beta_{i\bar{j}} -a.
\]
Thus, by the implicit function theorem, for all $t$ sufficiently small we can find a function $\phi_{t}\in C^{k,\alpha}$ and a constant $a_{t} \in \mathbb{R}$ such that
\[
F(\alpha+t\beta + \ddb\phi_{t}) = c(\omega, [\alpha]) + a_{t}.
\]
Fixing $k\gg 2$, the Schauder theory implies that $\phi_t \in C^{\infty}(X,\mathbb{R})$, and the result follows.
\end{proof}

We will now give the proof of our main theorem, and then spend the remainder of the paper discussing applications.   Suppose $(X,\omega)$ is a compact complex manifold with a Hermitian metric, and fix a concave elliptic operator $(f, \Gamma)$, and an ansatz $W$. Suppose that  ${\rm Sol}_{\omega}(f,\Gamma,  W)$ is not empty, and hence by Lemma~\ref{lem: open} it is open.  By Lemma~\ref{lem: well-defined} we get a function
\[
{\rm Sol}_{\omega}(f,\Gamma, W) \ni [\alpha]_W \longrightarrow c(\omega, [\alpha]_W).
\]
\begin{thm}\label{thm: main}
With notation as above, the function
\[
{\rm Sol}_{\omega}(f,\Gamma, W) \ni [\alpha]_W \longrightarrow c(\omega, [\alpha]_{W})
\]
is concave with respect to the natural linear structure on real $(1,1)$ forms.  Furthermore, in case $W\equiv 0$, assume that there is a smooth $(1,1)$ form $\beta$ for which we have $\frac{d^2}{dt^2}\big|_{t=0}c(\omega, [\alpha+t\beta]_{0}) =0$.  If $\alpha$ denotes the solution of $F(\alpha)=c(\omega, [\alpha]_{0})$, then there is a function $\psi$ so that
\[
D^{2}F(\alpha) (\beta+\ddb \psi, \beta+\ddb \psi) \equiv0
\]
where we regard $D^{2}F(\alpha)$ as a negative semi-definite bilinear form on $\Lambda^{1,1}$.
\end{thm}
\begin{proof}
Before beginning, we recall a result of Gauduchon \cite{Gaud}.  Let $h_{i\bar{j}}$ be any Hermitian metric on $X$, and let $\eta$ be the associated $(1,1)$ form.  Then there is a function $\sigma: X\rightarrow \mathbb{R}$ so that
\[
\ddb(e^{\sigma }\eta)^{n-1}=0
\]
That is, we can conformally rescale $h$ to a Gauduchon metric.  Fix $[\alpha]_{W}\in {\rm Sol}_{\omega}(f,\Gamma, W)$, and let $\alpha$ denote the solution of $F_{\omega}(\alpha)= c(\omega, [\alpha]_W)$.  Let $\beta$ be a smooth $(1,1)$ form.  By Lemma~\ref{lem: open}, for all $t$ sufficiently small there is a smooth function $\phi_t$ so that
\[
F_{\omega}(\alpha + t\beta + \ddb \phi_t + W(d\phi_t))= c(\omega [\alpha+t\beta]_{W})= c(t)
\]
Differentiating this equation in $t$ yields
\begin{equation}\label{eq: mainComp}
\begin{aligned}
\dot{c} &= F^{i\bar{j}}\left(\beta_{i\bar{j}} + \del_i\del_{\bar{j}}\dot{\phi} + W_{i\bar{j}}(d\dot{\phi}_{t})\right)\\
\ddot{c}&= F^{i\bar{j}}\del_i\del_{\bar{j}}\ddot{\phi} + F^{i\bar{j}}W_{i\bar{j}}(d\ddot{\phi}_{t})\\
&\quad + F^{i\bar{j},k\bar{\ell}}\left(\beta_{i\bar{j}} + \del_i\del_{\bar{j}}\dot{\phi} + W_{i\bar{j}}(d\dot{\phi}_t)\right)(\beta_{k\bar{\ell}} + \del_k\del_{\bar{\ell}}\dot{\phi} +W_{k\bar{\ell}}(d\dot{\phi}_t) ).
\end{aligned}
\end{equation}
Where $F^{i\bar{j},k\bar{\ell}}$ denotes the hessian of $F$.  Evaluating at $t=0$ and using the concavity of $F$ on $\Gamma$, we have
\begin{equation}\label{eq: concave}
\ddot{c}(0) \leq F^{i\bar{j}}\del_i\del_{\bar{j}}\ddot{\phi}+ F^{i\bar{j}}W_{i\bar{j}}(d\ddot{\phi}_{t}).
\end{equation}
Crucially, this equation holds at every point of $X$.  In particular, it holds at the maximum point of $\ddot{\phi}$.  Applying the maximum principle, we obtain
\[
\ddot{c}(0)\leq0
\]
which is the desired result.

Let us give another proof, which applies only in the case that $W=0$, but which yields a slightly stronger result. Suppose that $W=0$.  By the ellipticity of $F$, $F^{i\bar{j}}$ induces a Hermitian metric $h_{F}$ on $X$, with associated $(1,1)$ form $\eta_F$.  Hence we can apply Gauduchon's theorem to find $\sigma$ so that $e^{\sigma}\eta_{F}$ is Gauduchon.  Furthermore, after possibly adding a constant to $\sigma$, we can assume that $\int_{X}e^{(n-1)\sigma}\eta_{F}^n =1$.  Write 
\[
F^{i\bar{j}}\del_i\del_{\bar{j}}\ddot{\phi} = n\frac{\eta_{F}^{n-1}\wedge \ddb \ddot{\phi}}{\eta_{F}^{n}}.
\]
Multiply both sides of~\eqref{eq: concave} by $e^{(n-1)\sigma} \eta_{F}^{n}$ and integrate over $X$ to get
\[
\ddot{c}(0) = \ddot{c}(0)\int_{X}e^{(n-1)\sigma} \eta_{F}^{n} \leq n\int_{X} \ddb\ddot{\phi} \wedge (e^{\sigma}\eta_{F})^{n-1}.
\]
Integrating by parts, and using that $e^{\sigma}\eta_{F}$ is Gauduchon, we obtain $\ddot{c}(0)\leq 0$ which proves the result.  Note, that if
\[
\ddot{c}(0) =0
\]
then we must have equality throughout.  In particular, we have that
\[
F^{i\bar{j},k\bar{\ell}}\left(\beta_{i\bar{j}} + \del_i\del_{\bar{j}}\dot{\phi}\right)\left(\beta_{k\bar{\ell}} + \del_k\del_{\bar{\ell}}\dot{\phi}\right)\equiv 0
\]
and hence $\beta+\ddb \dot{\phi}$ lies in the kernel of the negative semi-definite bilinear form $D^2F : \Lambda^{1,1}\times\Lambda^{1,1} \rightarrow \mathbb{C}$.
\end{proof}

\begin{rk}\label{rk: obst}
The proof implies a type of ``ellipticity" for the function $c$ on the solvable set when $W=0$.  Namely, assuming $W=0$, suppose that there is a $L^{1}$ function $\psi$ so that $\beta+\ddb \psi \geq0$ in the sense of currents.  Then, using~\eqref{eq: mainComp} and the above notation we have
\[
\begin{aligned}
\dot{c}(0) = \dot{c}(0)\int_{X} e^{(n-1)\sigma} \eta_{F}^{n} &= n\int_{X} (\beta+\ddb\dot{\phi}) \wedge (e^{\sigma}\eta_{F})^{n-1}\\
&=  n\int_{X} \beta \wedge (e^{\sigma}\eta_{F})^{n-1}\\
&= n\int_{X} (\beta+\ddb\psi) \wedge (e^{\sigma}\eta_{F})^{n-1}\\
&\geq 0
\end{aligned}
\]
where in the third and fourth inequalities we used that $e^{\sigma}\eta_{F}$ is Gauduchon.  For example, when $[\beta]_0$ is the Bott-Chern class of an effective divisor, the inequality
\[
\frac{d}{dt} \big|_{t=0} c(\omega, [a+t\beta]_{0}) \geq 0
\]
can be regarded as a sort of algebraic obstruction for the existence of solutions to Problem~\ref{prob: genEll}.  In all the examples we have considered, the obstructions produced via this observation are clear ``by inspection".  Nevertheless, we hope this observation may be of some use in other settings.
\end{rk}

From Theorem~\ref{thm: main} we obtain generalized Khovanskii-Teissier type inequalities of the following form;  fix a complex manifold $(X,\omega)$ and an ansatz $W$, then for each $[\alpha]_{W}\in {\rm Solv}_{\omega}(f,\Gamma, W)$, and every $\beta\in\Lambda^{1,1}$ we have
\[
\frac{d^2}{dt^2}c(\omega, [\alpha]_{W}+t[\beta]_{W}) \leq 0.
\] 
The most useful applications of this result come from situations in which the ansatz $W=0$, and the constant $c(\omega, [\alpha]_W)$, as well as the solvability of the equation $(f, \Gamma, W)$, depend only on the algebraic structure of $X$.  That is, in situations where $c(\omega, [\alpha]_{W}) = c([\omega], [\alpha]_{W})$, and ${\rm Sol}_{\omega}(f,\Gamma,W)$ depends only on $[\omega], [\alpha]_{W}$, and can be determined by the algebraic properties of $X$, such as the intersection pairing on cohomology, and Poincar\'e duality. This is the case for the Monge-Amp\`ere equation, thanks to Yau's theorem \cite{Y}, which says that on a compact K\"ahler manifold $X$, with $W$=0, we have
\[
{\rm Sol}_{\omega}(\sigma_{n}^{1/n}, \Gamma_n)\cap H^{1,1}(X,\mathbb{R}) = \mathcal{K}, 
\]
and on this set we have
\[
c(\omega, [\alpha]) = \left(\frac{\int_{X}\alpha^n}{\int_{X}\omega^n}\right)^{\frac{1}{n}}.
\]
Furthermore, by the result of Demailly-P\u{a}un \cite{DP}, the K\"ahler cone $\mathcal{K}$ can be determined purely algebraically.  All together, these results imply restrictions on the intersection theory of a general K\"ahler manifold.

Even in the simplest case, when the ansatz $W=0$, producing algebraic conditions determining the solvability of an equation $(f,\Gamma)$ is a hard problem, and the subject of current research \cite{LS, CS, Sze, CY}.  On the other hand, we can at least attempt to give a structural condition under which the constant $c(\omega, [\alpha])$ can be shown to depend only on $[\omega], [\alpha]$ (and $\Gamma$!).  The remainder of this paper will be spent discussing explicit examples of equations to which Theorem~\ref{thm: main} applies. This will motivate the introduction of the structural condition alluded to above, which appears at the end of the paper.  We will use the following notation; for $(1,1)$ forms $\alpha_i, 1\leq i\leq n$ we denote
\[
\alpha_1.\alpha_2\cdots\alpha_{n-1}.\alpha_n= \int_{X}\alpha_1\wedge\alpha_2\cdots \wedge\alpha_{n-1}\wedge\alpha_n
\]
\begin{ex}[Hessian Equations]\label{ex: KahlerHess}
Let $(X,\omega)$ be a compact K\"ahler manifold, and $\alpha$ be a closed, real $(1,1)$ form on $X$, and $W=0$.  Then an exercise in linear algebra shows that
\begin{equation}\label{eq: hessConst}
\sigma_{k}(\lambda(\alpha)) = \binom{n}{k} \frac{\omega^{n-k}\wedge\alpha^k}{\omega^n}.
\end{equation}
If $(\binom{n}{k}^{-1}\sigma_{k}(\lambda(\alpha))^{\frac{1}{k}} =c$, and $\alpha$ is $k$-positive then, multiplying~\eqref{eq: hessConst} by $\omega^{n}$ and integrating over $X$ we must have
\[
c^{k} =\frac{\int_{X} \omega^{n-k}\wedge \alpha^{k}}{\int_{X}\omega^n}.
\]
When $k$ is odd, $c$ is uniquely specified, and when $k$ is even, $c$ is the positive root of the cohomological quantity appearing above. In particular, given that $\alpha$ is $k$-positive, there is unique constant $c$, determined only by $[\alpha],[\omega] \in H^{1,1}(X,\mathbb{R})$, for which Problem~\ref{prob: genEll} can possibly admit a solution.  Furthermore, by work of Dinew-Kolodziej \cite{DK} (see also Sz\'ekelyhidi \cite{Sze}), if $[\alpha]$ admits a $k$-positive representative, then the complex Hessian equation $(\binom{n}{k}^{-1}\sigma_{k}(\lambda(\alpha))^{\frac{1}{k}} =c$ admits a smooth solution in the class $[\alpha]$.  Thus, we conclude
\begin{cor}\label{cor: Hess}
Let $\mathcal{K}_{\Gamma_k, \omega}\subset H^{1,1}(X,\mathbb{R})$ denote the cone of classes admitting $k$-positive representatives with respect to $\omega$.  Then the function
\[
H^{1,1}(X,\mathbb{R}) \ni [\alpha] \longmapsto \left( \frac{\omega^{n-k}.\alpha^k}{\omega^n}\right)^{\frac{1}{k}}
\]
is concave on $\mathcal{K}_{\Gamma_k, \omega}$.  In particular, for any $[\beta] \in H^{1,1}(X,\mathbb{R})$ we have the generalized Khovanskii-Teissier type inequality
\begin{equation}\label{eq: KTHess}
\left(\omega^{n-k}.\alpha^{k-2}.\beta^2\right)\left(\omega^{n-k}.\alpha^k\right) \leq \left(\alpha^{k-1}.\beta.\omega^{n-k}\right)^2,
\end{equation}
with equality if and only if $[\beta] = \lambda[\alpha]$ for some $\lambda \in \mathbb{R}$.
\end{cor}
\begin{rk}
In fact, it is not hard to see that the concavity, and hence the generalized Khovanskii-Teissier type inequality, hold on the a priori larger set of classes $[\alpha]$ admitting a $k$-positive representative with respect to {\em some} K\"ahler metric in the class $[\omega]$.
\end{rk}

Corollary~\ref{cor: Hess} was known to Tosatti-Weinkove \cite{T} who observed that it followed from Dinew-Ko\l odziej's work \cite{DK}, and Demailly's technique \cite{Dem}.  It was also proved independently by Xiao \cite{Xia1}, who also proves polarized versions of~\eqref{eq: KTHess}. 

\begin{rk}
Applying the discussion of Remark~\ref{rk: obst} to the class of an effective divisor $[E]$ on $X$ yields that, for any class $[\alpha]$ admitting a $k$-positive representative with respect to $\omega$, we have
\[
\int_{E}\omega^{n-k}\wedge\alpha^{k-1} \geq 0.
\]
Note that this result, with strict inequality, follows from linear algebra.
\end{rk}

\end{ex}

\begin{ex}[Hessian Quotient Equations]\label{ex: hessQuoAlg}
Let $(X,\omega)$ be a compact K\"ahler manifold, and $\alpha$ be a closed, real $(1,1)$ form on $X$, and $W=0$. The considerations of Example~\ref{ex: KahlerHess} hold for the Hessian quotient equations of Example~\ref{ex: hessQuoPDE}.  Using~\eqref{eq: hessConst} we obtain
\[
c = \left(\frac{\int_{X}\omega^{n-k}\wedge\alpha^{k}}{\int_{X}\omega^{n-\ell}\wedge\alpha^{\ell}}\right)^{\frac{1}{k-\ell}}
\]
and hence there is a unique constant, depending only on $[\alpha], [\omega] \in H^{1,1}(X,\mathbb{R})$ such that the Hessian quotient equation on $(X,\omega)$, in the $k$-positive class $[\alpha]$ can admit a solution.  Applying Theorem~\ref{thm: main} we obtain
\begin{cor}\label{cor: hessQuo}
Let ${\rm Sol}_{\omega}(\frac{k}{\ell}, \Gamma_k) \subset \mathcal{K}_{\Gamma_k, \omega}$ denote the set of classes $[\alpha]$ admitting a $k$-positive representative solving the Hessian quotient equation
\[
\left(\frac{\binom{n}{k}^{-1}\sigma_{k}(\alpha)}{\binom{n}{\ell}^{-1}\sigma_{\ell}(\alpha)}\right)^{1/(k-\ell)} = c.
\]
Then on ${\rm Sol}_{\omega}(\frac{k}{\ell}, \Gamma_k)$, the function
\[
H^{1,1}(X,\mathbb{R}) \ni [\alpha] \longmapsto \left(\frac{\omega^{n-k}.\alpha^{k}}{\omega^{n-\ell}.\alpha^{\ell}}\right)^{\frac{1}{k-\ell}}
\]
is concave.
\end{cor}

The concavity implies a Khovanskii-Teissier type inequality, which is not difficult to compute, but rather long to write down.  For simplicity, we make the following observation; if $g$ is a positive, concave function, and $a\geq 1$, then $g^{-a}$ is convex.  We conclude that
\[
H^{1,1}(X,\mathbb{R}) \ni [\alpha] \longmapsto \frac{\omega^{n-\ell}.\alpha^{\ell}}{\omega^{n-k}.\alpha^{k}}
\]
is convex on ${\rm Solv}_{\omega}(\frac{k}{\ell}, \Gamma_k)$; in particular, for every  $[\beta] \in H^{1,1}(X,\mathbb{R})$ we have
\[
\begin{aligned}
&\ell(\ell-1)(\omega^{n-\ell}.\alpha^{\ell-2}.\beta^2)(\omega^{n-k}.\alpha^k)-k(k-1)(\omega^{n-k}.\alpha^{k-2}.\beta^2)(\omega^{n-\ell}.\alpha^{\ell})\\
& \geq 2k\left(\ell (\omega^{n-\ell}.\alpha^{\ell-1}.\beta) - k(\omega^{n-k}.\alpha^{k-1}.\beta)\frac{(\omega^{n-\ell}.\alpha^{\ell})}{(\omega^{n-k}.\alpha^k)}\right)(\omega^{n-k}.\alpha^{k-1}.\beta)
\end{aligned}
\]

\begin{rk}
Sz\'ekelyhidi \cite{Sze}, and Lejmi-Sz\'ekelyhidi \cite{LS}  have made conjectures concerning necessary and sufficiently algebraic conditions for the Hessian quotient equations to be solvable.  For effective divisors, the obstructions appearing in \cite{Sze, LS} are precisely those obtained from Remark~\ref{rk: obst}.  When $k=n$ and $\ell=n-1$, and $(X,\omega)$ is toric, this conjecture has been proved by the author and Sz\'ekelyhidi \cite{CS}.  Combining this work with Corollary~\ref{cor: hessQuo} we obtain constraints on the intersection theory of toric varieties. 
\end{rk} 
\end{ex}

\begin{ex}[Lagrangian Phase Operator]\label{ex: LagPhaseAG}
Again we consider $(X,\omega)$ compact K\"ahler manifold, $\alpha$ a closed, and real $(1,1)$ form on $X$, and $W=0$.  The Lagrangian phase operator considered in Example~\ref{ex: lagPhasePDE} satisfies a slightly more complicated version of the principle discussed in the previous two examples.  Consider the volume form
\[
\Omega(\omega, \alpha) = \sqrt{\left(\prod_{i=1}^{n} 1+\lambda_i(\alpha)^2\right)} \omega^{n}.
\]
It is straightforward to show that
\[
-(-\sqrt{-1})^{n}\int_{X}e^{\sqrt{-1}\Theta_{\omega}(\alpha)}\Omega(\omega, \alpha) = -\int_{X}e^{(\sqrt{-1}\omega+\alpha)}.
\]
For simplicity, let us denote $Z([\alpha], [\omega]) =  -\int_{X}e^{(\sqrt{-1}\omega+\alpha)}$.  If $\Theta_{\omega}(\alpha) = c$ on $X$, then we clearly have
\[
-\int_{X} (\sqrt{-1}\omega+\alpha)^{n} \in \mathbb{R}_{>0} e^{\sqrt{-1}c- (n-2)\frac{\pi}{2}}
\]
and so, in particular, the constant $c$ appearing on the right-hand side of the Lagrangian phase equation is determined modulo $2\pi$ by the classes $[\alpha], [\omega] \in H^{1,1}(X,\mathbb{R})$ on the solvable set. Furthermore, if $\alpha$ solves the Lagrangian phase equation with $c\geq (n-2)\frac{\pi}{2}$, then we have
\[
c = (n-2)\frac{\pi}{2} + {\rm Arg}_{p.v}(Z([\alpha],[\omega]))
\]
where ${\rm Arg}_{p.v.}$ denotes the principal value of argument. For each angle $c\in (-n\frac{\pi}{2}, n\frac{\pi}{2})$, let 
\[
\Gamma_{\geq c} = \{ (\lambda_1, \ldots, \lambda_n) \in \mathbb{R}^{n} : \sum_{i=1}^{n}\arctan(\lambda_i) \geq c \}.
\]
With this discussion, Theorem~\ref{thm: main} implies
\begin{cor}
Let ${\rm Sol}_{\omega}(\Theta, \Gamma_{(n-1)\frac{\pi}{2}}) \subset \mathcal{K}_{\Gamma_{(n-1)\frac{\pi}{2}}, \omega}$ denote the set of classes $[\alpha] \in H^{1,1}(X,\mathbb{R})$ admitting solutions of the Lagrangian phase equation $\Theta(\alpha)=c$ with $c\geq (n-1)\frac{\pi}{2}$.  Then, on this set
\[
{\rm Arg}_{p.v}(Z([\alpha],[\omega]))
\]
defines a concave map into $[\frac{\pi}{2}, \pi)$.  Furthermore, for every $\delta\in (0,\frac{\pi}{2})$, there is a constant $C=C(\delta)>0$, such that
\[
-\exp[-C{\rm Arg}_{p.v}(Z([\alpha],[\omega]))]
\]
is concave on ${\rm Sol}_{\omega}(\Theta, \Gamma_{(n-2)\frac{\pi}{2}+\delta})$.
\end{cor}
\begin{rk}
The Lagrangian phase equation, also called the deformed Hermitian-Yang-Mills (dHYM) equation, has recently seen a great deal of interest, owing to its importance in mirror symmetry \cite{MMMS, LYZ, CXY, CJY, JY}. In particular, the author, with Jacob and Yau \cite{CJY}, and the author and Yau \cite{CY} have formulated conjectures relating algebraic geometry and the solvability of the dHYM which would, in principle, determine the set ${\rm Sol}_{\omega}(\Theta, \Gamma_{(n-1)\frac{\pi}{2}})$ in terms of algebraic data.
\end{rk}
\end{ex}

So far, our discussion has focused on K\"ahler examples.  We now turn to some non-K\"ahler applications.

\begin{ex}
Let $(X,\omega)$ be a compact complex manifold with $\dim_{\mathbb{C}}X =n\geq 2$.  Recall that the $(n-1,n-1)$ {\em Aeppli} cohomology group of $X$ is defined by 
\[
H^{n-1,n-1}_{A}(X, \mathbb{R}) = \frac{\{ \del\dbar \text{-closed, real $(n-1, n-1)$ forms} \}}{\{\del \gamma + \overline{\del \gamma} : \gamma \in \Lambda^{n-2,n-1}\}}
\]
Note that every Gauduchon metric $\alpha$ on $X$ induces a class $[\alpha^{n-1}]_{A} \in H_{A}^{n-1,n-1}(X,\mathbb{R})$.  We have the following  

\begin{thm}[Sz\'ekelyhidi-Tosatti-Weinkove \cite{STW}]\label{thm: STW}
Let $(X,\omega)$ be a compact complex manifold with a fixed Gauduchon metric $\omega$.  Let $\alpha_0$ be any Gauduchon metric.  Then there exists a unique function $\phi: X\rightarrow \mathbb{R}$, with $\sup_{X}\phi=0$, and a unique constant $c\in\mathbb{R}_{>0}$, and Gauduchon metric $\alpha_{\phi}$, satisfying
\begin{equation}\label{eq: TWansatz}
\alpha_{\phi}^{n-1} = \alpha_0^{n-1} + \ddb \phi \wedge \omega^{n-2} + {\rm Re}(\sqrt{-1}\del \phi\wedge \dbar(\omega^{n-2}))
\end{equation}
and solving
\begin{equation}\label{eq: GaudMonge}
\alpha_{\phi}^{n} = c\omega^{n}.
\end{equation}
\end{thm}
Observe that, in the above theorem, we clearly have $[\alpha_0^{n-1}]_{A}= [\alpha^{n-1}]_{A}$.  A few words are in order about how this theorem is proved.  In \cite{TW, Pop} it is shown that, under the ansatz~\eqref{eq: TWansatz}, equation~\eqref{eq: GaudMonge} can be reduced to the equation
\[
f(\lambda) = ((n-1)c)^{\frac{1}{n}}
\]
where $f$ is the $(n-1,n-1)$ Monge-Amp\`ere equation discussed in Example~\ref{ex: n-1Eq}, and $\lambda$ are the eigenvalues of a certain {\em related} Hermitian $(1,1)$ form.  It will be important for us how this particular Hermitian form is constructed.  We say that a real $(n-1,n-1)$ form $\Omega$ is Gauduchon if $\Omega$ is positive, and $\ddb \Omega=0$. By a theorem of Michelsohn \cite{Mic} every Gauduchon $(n-1,n-1)$ form $\Omega$ satisfies $\Omega= \alpha^{n-1}$ for a unique Gauduchon metric $\alpha$.  Let $\mathcal{G}$ denote the set of Gauduchon $(n-1,n-1)$ forms, which is a convex cone in the space of real $(n-1,n-1)$ forms on $X$.  Following \cite{TW}, to each $\Omega$ we associate a smooth, not necessarily positive, $(1,1)$ form $\alpha_{\Omega}$ by
\[
(\alpha_{\Omega})_{i\bar{j}} = {\rm Tr}_{\omega}\left( \frac{1}{(n-1)!} *_{\omega}\Omega\right)\omega_{i\bar{j}} - \left(\frac{1}{(n-2)!} *_{\omega}\Omega\right).
\]
Then, imposing~\eqref{eq: TWansatz}, and solving~\eqref{eq: GaudMonge} is equivalent to solving the $(n-1,n-1)$ Monge-Amp\`ere equation in Example~\ref{ex: n-1Eq} with $(1,1)$ form
\[
(\alpha_{\Omega})_{i\bar{j}} + \phi_{i\bar{j}} + W_{i\bar{j}}(d\phi)
\]
where $W$ is a certain {\em anstaz} in the sense of Definition~\ref{def: ansatz}.  The corresponding Gauduchon metric is obtained by applying the $P_{\omega}$ operator described in Example~\ref{ex: n-1Eq}, which acts on $(1,1)$-forms by
\[
P_{\omega}(\beta) = \frac{1}{n-1}\left[\left({\rm Tr}_{\omega}\beta\right)\omega-\beta\right].
\]
A straightforward computation shows that
\begin{equation}\label{eq: STWmet}
\alpha_{\phi} = \frac{1}{(n-1)!}*_{\omega}\Omega + \frac{1}{n-1}\left(\Delta_{\omega}\phi - \ddb\phi\right) + P_{\omega}(W_{i\bar{j}}(d\phi))
\end{equation}

Observe that the linear structure on Gauduchon $(n-1,n-1)$ forms is mapped to the linear structure on $(1,1)$ forms by this construction.  Namely, if $\Omega_i$ are Gauduchon $(n-1,n-1)$ forms, for $i=0,1$, and $\Omega_t = (1-t)\Omega_0 + t\Omega_1$,  then applying the above construction yields
\[
\alpha_{\Omega_t} = (1-t)\alpha_{\Omega_0} + t\alpha_{\Omega_1}.
\]
Thus, by Theorem~\ref{thm: main} we obtain
\begin{cor}\label{cor: Aeppli}
Let $(X,\omega)$ be a compact Hermitian manifold, and let $\mathcal{G}$ denote the space of Gauduchon $(n-1,n-1)$ forms.  For each $\Omega$, let $\alpha_{\Omega, CY}$ be the Gauduchon metric produced by Theorem~\ref{thm: STW}, and define
\[
{\rm Vol}^{1/n}(\Omega) = \left(\frac{\int_{X}\alpha_{\Omega, CY}^n}{\int_{X}\omega^{n}}\right)^{1/n}.
\]
Then ${\rm Vol}^{1/n}$ defines a concave function on the space $\mathcal{G}$.
\end{cor}

\begin{rk}
Observe that Corollary~\ref{cor: Aeppli} also holds if we consider only Gauduchon $(n-1,n-1)$ forms with fixed Aeppli cohomology class.  In particular, it would be interesting to know if ${\rm Vol}^{1/n}$ achieves a finite maximum on Gauduchon forms in a fixed class, as this could be interpreted as an invariant of the Aeppli cohomology class.  It would furthermore be interesting to understand whether this maximum (provided it exists) is unique.
\end{rk}

\begin{rk}
As pointed out in \cite{STW}, Theorem~\ref{thm: STW} also holds for the $(n-1,n-1) -\sigma_k$ equations, and hence the above construction can be extended to define mixed volume functions on $\mathcal{G}$, which are also concave, in complete analogy with Example~\ref{ex: hessQuoAlg}
\end{rk}

\end{ex}

\begin{ex}
Let $(X,\omega)$ be a Hermitian manifold, and let $\alpha$ be any smooth Hermitian metric.  Then a result of Tosatti-Weinkove \cite{TW2} says that there exists a unique function $\phi:X\rightarrow \mathbb{R}$ such that $\alpha_{\phi} : \alpha+\ddb \phi$ solves the Monge-Amp\`ere equation
\[
\alpha_{\phi}^{n} = e^{b} \omega^{n}
\]
for some constant $b = b(\alpha,\omega)$. Applying Theorem~\ref{thm: main}  we conclude that the map $\alpha \mapsto e^{b/n}$, defined on $\ddb$-equivalence classes of Hermitian metrics, is concave.  This result may have applications to the study of compact complex manifolds admitting big and nef $(1,1)$ classes; see \cite{TS}.
\end{ex}

Inspired by examples~\ref{ex: KahlerHess},~\ref{ex: hessQuoAlg} and~\ref{ex: LagPhaseAG} we make the following definition.

\begin{defn}
We say that an elliptic operator $(f,\Gamma)$  is of {\em cohomological type} if it has the following property:
\begin{enumerate}
\item There is a finite-to-one map $\phi_{f}: \mathbb{R} \rightarrow \mathbb{R}$ such that $\phi_{f}:f(\Gamma) \rightarrow [0,\pi]$ is injective.
\item There is a map $\hat{Z}:\mathbb{R} \rightarrow \mathbb{C}^*$ such that
\[
\widehat{Z}(x) \in \mathbb{R}_{>0}e^{\sqrt{-1}\phi_{f}(x)}.
\]
\item For each admissible $\alpha \in [\alpha]$ there is a volume form $\Omega(\omega, \alpha)$ such that
\[
\int_{X}\widehat{Z}(F_{\omega}(\alpha)) \Omega(\omega, \alpha) =: Z_{F}([\alpha],[\omega]) \in \mathbb{C}
\]
depends only on the classes $[\alpha],[\omega]$.
\end{enumerate}
\end{defn}

It is not hard to see that all of the elliptic operators in examples~\ref{ex: KahlerHess},~\ref{ex: hessQuoAlg} and~\ref{ex: LagPhaseAG} can be fit into this framework.  For example, for the Hessian equation $f(\lambda) = \sigma_{k}^{1/k}$, we can define
\[
\hat{Z}_{k}(x)= (1+\sqrt{-1}x^k)\in \mathbb{C}^{*}.
\]
In particular, we can take $\phi(x) = \arctan(x^k)$.  We define the volume form to be $\Omega_{k}(\alpha,\omega) = \omega^{n}$.  Then we have
\[
\int_{X}\hat{Z}_{k}(\sigma_{k}^{1/k}(\alpha))\omega^{n} = \int_{X} \omega^{n} + \sqrt{-1}\omega^{n-k}\wedge\alpha^{k} = Z_{k}([\alpha],[\omega]).
\]
The Hessian quotient equations of Example~\ref{ex: hessQuoAlg} can be treated similarly, and we have essentially already described how the Lagrangian phase operator can be fit into this framework.  The trivial observation is that when $(f,\Gamma)$ is a cohomological type, the function $c(\omega, [\alpha])$ can be determined by cohomology (together with $\Gamma$).

\begin{lem}
Suppose $(f,\Gamma)$ is of cohomological type.  Then, on the set  of classes $K_{\Gamma, \omega} \subset H^{1,1}(X,\mathbb{R})$ admitting $\Gamma$-admissible representatives, the constant $c$ depends only depends only on $[\alpha], [\omega]$, and is determined by
\[
Z_{F}([\alpha], [\omega]) \in \mathbb{R}_{>0}e^{\sqrt{-1}\phi_{f}(c)}
\]
\end{lem}
\begin{proof}
By assumption, for any class $[\alpha]$ admitting a $\Gamma$-admissible representative we have
\[
Z_{F}([\alpha], [\omega]) = \int_{X}\widehat{Z}(F_{\omega}(\alpha)) \Omega(\omega, \alpha)
\]
On the other hand, if $[\alpha]$ admits a $\Gamma$-admissible solution of $F=c$ then we have $\hat{Z}(F_{\omega}(\alpha)) \in \mathbb{R}_{>0}e^{\sqrt{-1}\phi_{f}(c)}$.  Thus,
\[
Z_{F}([\alpha], [\omega]) \in \mathbb{R}_{>0}e^{\sqrt{-1}\phi_{F}(c)}.
\]
On the other hand, combining the assumption that $\phi_{f} : f(\Gamma)\rightarrow [0,\pi]$ is injective, with that fact that $c\in f(\Gamma)$, it follows that $c$ is uniquely specified by $\phi_{f}(c)$.
\end{proof}

Combining this simple lemma with Theorem~\ref{thm: main} we make the following conclusion.  Let $(X,\omega)$ be a K\"ahler manifold.  Any concave elliptic operator $(f,\Gamma)$ of cohomological type induces a concave function on the subset of $H^{1,1}(X,\mathbb{R})$ consisting of classes containing admissible solutions of $f=c$ for some K\"ahler metric in the class $[\omega]$.


\begin{thebibliography}{99}

\bibitem{Bou} S. Boucksom, {\em On the volume of a line bundle}, Internat. J. Math. {\bf 13} (2002), no. 10, 1043--1063.

\bibitem{CNS} L. Caffarelli, L. Nirenberg, and J. Spruck {\em The Dirichlet problem for nonlinear second order elliptic equations, III: Functions of the eigenvalues of the Hessian}, Acta Math. {\bf 155} (1985), 261--301.

\bibitem{Chen04} X.-X. Chen {\em A new parabolic flow in K\"ahler manifolds} Comm. Anal. Geom. {\bf 12} (2004), no. 4, 837--852.

\bibitem{CJY} T. C. Collins, A. Jacob, and S.-T. Yau {\em $(1,1)$ forms with specified Lagrangian phase: A priori estimates and algebraic obstructions}, arXiv:1508.01934

\bibitem{CPW} T. C. Collins, S. Picard, and X. Wu {\em Concavity of the Lagrangian phase operator and applications}, Calc. Var. Partial Differential Equations {\bf 56} (2017), no. 4, Art. 89

\bibitem{CS} T. C. Collins, and G. Sz\'ekelyhidi {\em Convergence of the $J$-flow on toric manifolds} J. Differential Geom. {\bf 107} (2017), no. 1, 81--109.

\bibitem{CXY} T. C. Collins, D. Xie, and S.-T. Yau {\em The deformed Hermitian-Yang-Mills equation in Geometry and Physics}, in Geometry and Physics: Volume I: A Festschrift in honour of Nigel Hitchin, Oxford University Press, 2018

\bibitem{CY} T. C. Collins, and S.-T. Yau {\em Moment maps, nonlinear PDE, and stability in mirror symmetry}, arXiv:1811.04824

\bibitem{Dem} J.-P. Demailly, {\em A numerical criterion for very ample line bundles}, J. Differential Geom. {\bf 37} (1993), no. 2, 323--374

\bibitem{DP} J.-P. Demailly, and M. P\u{a}un, {\em Numerical characterization of the K\"ahler cone of a compact K\"ahler manifold}, Ann. of Math. (2) {\bf 159} (2004), no. 3, 1247--1274.

\bibitem{DK} S. Dinew, and S. Ko\l odziej {\em Liouville and Calabi-Yau type theorems for complex Hessian equations}, Amer. J. Math. {\bf 139} (2017), no. 2, 403--415.

\bibitem{Do} S.K. Donaldson, {\em Moment maps in differential geometry}, Surveys in differential geometry, Vol. VIII (Boston, MA, 2002), 171--189, Surv. Differ. Geom., {\bf 8}, Int. Press, Somerville, MA, 2003.

\bibitem{ELMNP} L. Ein, R. Lazarsfeld, M. Musta\c{t}\u{a}, M. Nakamaye, and M. Popa {\em Asymptotic invariants of line bundles}, Pure Appl. Math. {\bf 1} (2005), 379--403.

\bibitem{Gar} L. G\aa rding {\em An inequality for hyperbolic polynomials}, J. Math. Mech. {\bf 8} (1959), 957--965.

\bibitem{Gaud} P. Gauduchon, {\em Le th\'eor\`eme de l'excentricit\'e nulle}, C. R. Acad. Sci. Paris S\'er. A-B, {\bf 285} (1977), A387--A390.

\bibitem{Gro} M Gromov, {\em Convex sets and K\"ahler manifolds}, Advances in Differential Geometry and Topology, 1--38, World Sci. Publ., Teaneck, NJ, 1990.

\bibitem{JY} A. Jacob, and S.-T. Yau {\em A special Lagrangian type equation for holomorphic line bundles}, Math. Ann. {\bf 369} (2017), no.1--2, 869--898

\bibitem{Kho} A. G. Khovanski\u{i}, {\em Analogues of Alexandrov-Fenchel inequalities for hyperbolic forms} Dokl. Aka. Nauk SSSR {\bf 276} (1984), no. 6, 1332--1334.

\bibitem{Laz} R. Lazarsfeld {\em Positivity in Algebraic Geometry. I\&II}, Ergebnisse Math. Grenzg. {\bf 49 \& 49}, Springer, 2004.

\bibitem{LM} R. Lazarsfeld, and M. Musta\c{t}\u{a} {\em Convex bodies assoiciated to linear series}, Ann. Sci. \'Ec. Norm. Sup\'er. (4) {\bf 42} (2009), no. 5, 783--835.

\bibitem{LS} M. Lejmi, and G. Sz\'ekelyhidi {\em The $J$-flow and stability}, Adv. Math. {\bf 274} (2015), 404--431.

\bibitem{LYZ} N. C. Leung, S.-T. Yau, and E. Zaslow {\em From special Lagrangian to Hermitian-Yang-Mills via Fourier-Mukai}, Adv. Theor. Math. Phys. {\bf 4} (2000), no. 6, 1319--1341

\bibitem{MMMS} M. Mari\~no, R. Minasian, G. Moore, and A. Strominger {\em Nonlinear instantons from supersymmetric $p$-branes}, J. High Energy Phys. (2000), no. 1

\bibitem{Mic} M. L. Michelsohn, {\em On the existence of special metrics in complex geometry} Acta Math., {\bf 149} (1982), 261--295.

\bibitem{Pop} D. Popovici {\em Aeppli cohomology classes associated with Gauduchon metrics on compact complex manifolds}, Bull. Soc. Math. France {\bf 143} (2015), 763--800.

\bibitem{Spr} J. Spruck {\em Geometric aspects of the theory of fully nonlinear elliptic equations}, in Global theory of minimal surfaces, vol. 2, Amer. Math. Soc., Providence, RI, 2005, 283--309

\bibitem{Sze} G. Sz\'ekelyhidi, {\em Fully non-linear elliptic equations on compact Hermitian manifolds} J. Differential Geom. {\bf 109} (2018), no. 1, 81--109.

\bibitem{STW}  G. Sz\'ekelyhidi, V. Tosatti, and B. Weinkove {\em Gauduchon metrics with prescribed volume form}, Acta Math. {\bf 219} (2017), no. 1, 181--211

\bibitem{Teis1} B. Teissier, {\em Sur une in\'egalit\'e \`a la Minkowski pour les multiplciti\'es}, Ann. of Math. (2) {\bf 106} (1977), no. 1, 19--44.

\bibitem{Teis2} B. Teissier, {\em Bonnesen-type inequalities in algebraic geometry. I. Introduction to the problem}, Seminar on Differential Geometry, pp. 85--105, Ann. of Math. Stud., 102, Princeton Univ. Press, Princeton, NJ, 1982.

\bibitem{Teis3} B. Teissier, {\em Du th\'eor\`eme de l'index de Hodge aux in\'egalit\'es isop\'erim\'etriques}, C. R. Acad. Sci. Paris S\'er. A-B {\bf 288} (1979), no. 4, A287--A289.

\bibitem{T} V. Tosatti, personal communication

\bibitem{TS} V. Tosatti, {\em The Calabi-Yau theorem and K\"ahler currents}, Adv. Theor. Math. Phys. {\bf 20} (2016), no. 2, 381--404.

\bibitem{TW} V. Tosatti, and B. Weinkove {\em Hermitian metrics, $(n-1,n-1)$ forms, and Monge-Amp\`ere equations}, to appear in J. Reine Angew. Math.

\bibitem{TW2} V. Tosatti, and B. Weinkove {\em The complex Monge-Amp\`ere equation on compact Hermitian manifolds}, J. Amer. Math. Soc. {\bf 23} (2010), no. 4, 1187--1195.

\bibitem{Xia1} J. Xiao, {\em Hodge-Index type inequalities, hyperbolic polynomials and complex hessian equations}, preprint, 2018, arXiv:1910.04662v1

\bibitem{Xia2} J. Xiao, {\em Mixed Hodge-Riemann bilinear relations and $m$-positivity}, preprint, 2018, arXiv:1811.05865v1

\bibitem{Y} S.-T. Yau {\em On the Ricci curvature of a compact K\"ahler manifolds and the complex Monge-Amp\`ere equation, I}, Comm. Pure. Appl. Math {\bf 31} (1978), no. 3, 339-411.

\bibitem{Yuan} Y. Yuan {\em Global solutions to special Lagrangian equations} Proc. Amer. Math. Soc. {\bf 134} (2006), no. 5, 1355--1358.

\end{thebibliography}
\end{document}